\documentclass[11pt]{article}
\usepackage{graphicx}
\usepackage{amssymb}
\usepackage{epstopdf}
\usepackage{amsthm}   
\newcommand{\bC}{{\Bbb C}}

\newcommand{\bR}{{\Bbb R}}
\newcommand{\bN}{{\Bbb N}}
\newcommand{\bA}{{\Bbb A}}

\newcommand{\bp}{{\bf p}}
\newcommand{\bq}{{\bf q}}

\newcommand{\spec}{{\rm{Spec}}}

\newcommand{\val}{{\rm{val}}}

\let \cedilla =\c

\renewcommand{\o}[0]{{\mathcal O}} 

\newcommand{\ta}{{\widetilde{\a}}}
\renewcommand{\a}{{\frak{a}}}
\renewcommand{\b}{{\frak{b}}}

\newcommand{\lct}{{\rm{lct}}}
\newcommand{\mld}{{\rm{mld}}}

\newcommand{\wA}{{\widehat{A}}}
\newcommand{\ha}{{\widehat{\a}}}
\newcommand{\mult}{{\rm{mult}}}
\newcommand{\m}{{\frak{m}}}

\def\to {\longrightarrow}

\newtheorem{thm}{Theorem}[section]

\newtheorem{lem}[thm]{Lemma}
\newtheorem{cor}[thm]{Corollary}
\newtheorem{prop}[thm]{Proposition}

\theoremstyle{definition}

\newtheorem{defn}[thm]{Definition}

\newtheorem{exmp}[thm]{Example}
\newtheorem{conj}[thm]{Conjecture}
 
\newtheorem{rem}[thm]{Remark}

\title{ The minimal log discrepancies on a smooth surface\\
 in positive characteristic}
\author{Shihoko Ishii}
\begin{document}
\date{}

\maketitle
\footnote{The author is partially supported by JSPS 19K03428}

\begin{abstract} 
This paper shows that Musta\cedilla{t}\v{a}-Nakamura's conjecture holds 
for pairs consisting of a smooth surface and a multiideal with a real exponent
 over the base field of positive 
characteristic.
As corollaries, we  obtain the ascending chain condition of the minimal
log discrepancies and of the log canonical thresholds for those pairs.
We also obtain finiteness of the set of the minimal log discrepancies of those pairs
for a fixed real exponent.
\end{abstract}
\section{Introduction}
   Two invariants, the minimal log discrepancies and the log canonical thresholds, for a 
    singularity of a pair consisting of a variety and a multiideal with a real exponent  
    play important roles in   
    birational geometry.
     For example, the ascending chain condition (ACC, for short) of these invariants 
     under certain conditions would 
     give a significant step in a proof of the Minimal Model Problem (MMP, for short).
     ACC Conjecture for log canonical thresholds over the base field of characteristic 0 
     is proved in \cite{hmx}, 
     while  not yet  over the base field of positive characteristic.
     On the other hand, ACC Conjecture 
     for  minimal log discrepancies 
     is not proved even in characteristic 0. \\    
  \ \ 
     Musta\cedilla{t}\v{a} and Nakamura (\cite{mn}) posed a conjecture, 
     say Musta\cedilla{t}\v{a}-Nakamura's conjecture (MN Conjecture, for short) and proved
     that it implies ACC Conjecture for the minimal log discrepancies 
     for the base field of characteristic 0.
     Then, Kawakita (\cite{kawk2}) proved the converse also holds for dimension 3 in characteristic 0.
     In the same paper he proved ACC for dimension 2 in characteristic 0.
     Aside ACC, MN Conjecture plays important roles on basic properties of singularities
     (openess of good singularities, stability of good singularities under a deformation, etc.,
     see, for example \cite{findet}).\\
\ \      
     In this paper, we focus on pairs $(A, \a^e)$ consisting of a smooth surface $A$ 
     and  a  non-zero multiideal $\a^e=\a_1^{e_1}\cdots\a_s^{e_s}$ 
     ($\a_i$'s are non-zero coherent ideal sheaves on $A$) with an exponent
     $e=\{e_1,\ldots,e_s\}$ $(e_i\in \bR_{>0})$ on $A$ defined over an algebraically closed 
     base field $k$ of arbitrary characteristic.\\
 \ \     
          As our interest is in a smooth surface $A$, 
     we state the conjectures on smooth varieties, although the primary conjectures are 
     stated under more general settings:
     
 \begin{conj}[MN Conjecture] 
  Let $A$ be a smooth variety of dimension $N$ defined over 
     an  algebraically closed field $k$ and let $0\in A$ be a closed point. 
     Given a finite subset 
     $e\subset \bR_{>0}$, there
     is a positive integer $\ell_{N,e}$ (depending on $N$ and $e$) such that for every multiideal 
     $\a^e$ on $A$ with the
    exponent  $e$, there is a prime divisor $E$ that computes $\mld(0; A,  \a^e)$ and 
    satisfies
    $k_E \leq \ell_{N,e}$.
 \end{conj}    
 
 \begin{conj}[ACC Conjecture for mld]
   Let $A$, $N$, $e$ and $0$   be as above. 
  For every fixed DCC set $J \subset \bR_{>0}$, the set
$$\{\mld(0; A, \a^e) \mid  e\subset J, (A, \a^e)\ \mbox{is\ log\ canonical\ at\ }
      0\}  $$
  satisfies ascending chain condition (ACC).     
 
 \end{conj}
 
 We know the following relations between the two conjectures in characteristic $0$.
 
 \begin{prop}[\cite{mn}, {\cite{kawk2}}] \label{equiv}
 Let $A$ be a smooth variety of dimension $N$ 
 defined over an algebraically closed field of characteristic $0$ 
 and 
 $0\in A$ a closed point. 
 For $N=2$, the both MN Conjecture and ACC conjecture hold.
 
  For general $N$, if the following (i) holds, then (ii) holds:
\begin{enumerate}
   \item[$\rm (i)$] MN Conjecture in the set of  pairs $(A, \a^e)$  holds for every finite set $e\subset \bR_{>0}$;
   \item[$\rm (ii)$] ACC Conjecture in the set of  pairs $(A, \a^e)$ holds for every DCC set $J$.
\end{enumerate} 
When $N=3$, the converse also holds.
 \end{prop}
 
 \ \   When the base field $k$ is an algebraically closed field of characteristic 0,
    it is proved in \cite{mn} that 
    MN Conjecture holds for the set of the pairs $(\bA_k^N, \a^e)$ for arbitrary 
    $N\geq 1$ with monomial ideals $\a_i$,
    as well as the set of the pairs $(A, \a^e)$ for a surface $A$ with an arbitrary non-zero multiideal
    $\a^e$ .
    Therefore, by Proposition \ref{equiv}, ACC Conjecture also holds for these classes in
    characteristic $0$.\\
\ \     
        When the base field $k$ is of positive characteristic, 
    MN Conjecture holds for the set of the pairs $(\bA_k^N, \a^e)$ for arbitrary 
    $N\geq 1$ and monomial ideals $\a_i$ (see, Lemma \ref{toric} and \cite[Corollary 1.10 ]{ii}).
    The proof of MN Conjecture for the set of pairs $(A,\a^e)$ with
    surfaces $A$ in characteristic $0$ in \cite{mn} uses 
    generic smoothness, therefore the proof does not work directly for positive characteristic case.
    By making use of the result for monomial multiideal case and Kawakita's result \cite{kawk1},
     we obtain the main result of this paper:
\begin{thm}\label{mainthm}
   Let $A$ be a smooth surface defined over an algebraically closed field of arbitrary  
   characteristic
   and $0\in A$ a closed point.
   Then, we obtain the following:
\begin{enumerate}
\item[$\rm(i)$]   
    MN Conjecture holds in the set of the pairs $(A, \a^e)$ for multiideal $\a^e$ 
   with an exponent $e\subset \bR_{>0}$.
   {\it I.e.}, There exists a positive number $\ell_e$ (depending only on $e$) 
   such that for every multiideal $\a^e$ with the exponent $e$
   there is    a prime divisor $E$ over $A$ which computes $\mld(0; A,\a^e)$
   and satisfies $k_E\leq \ell_e$.
\item[$\rm(ii)$] Moreover, 
    for every pair $(A, \a^e)$,  there is a monomial multiideal $ \a_*^e$
    on $\bA_k^2$   such that 
  $$\mld(0; A, \a^e)=\mld(0; \bA_k^2,\a_*^e)=\mld(0; \bA_\bC^2,\ta_*^e),$$
  where $\ta_*^e$ is the monomial multiideal on $\bA_\bC^2$ whose ideals $\ta_{*i}$'s
  are generated
  by the same monomial generators as of $\a_{*i}$'s.

\end{enumerate}
\end{thm} 

   On the way to prove the theorem, we generalize the result in \cite{kawk1} into 
   the positive characteristic case (Lemma \ref{weighted}) and as its application 
   we obtain the following:
   
\begin{thm}\label{weightedlct}
    Let $A$ be a smooth surface over an algebraically closed base field $k$ of
   arbitrary characteristic and $\a^e $ be a multiideal with a real exponent $e$ on $A$.
   Then, every exceptional prime divisor computing the log canonical threshold $\lct(0; A,\a^e) $ is 
    obtained by a weighted blow up.
    
    Moreover, 
   for every pair $(A, \a^e)$ over $k$
    such that $\lct(0; A, \a^e)$ is computed by an exceptional divisor, 
    there is a monomial multiideal $ \a_*^e$
    on $\bA_k^2$   such that 
    $$\lct(0; A, \a^e)=\lct(0; \bA_k^2, \a_*^e)=\lct(0; \bA_\bC^2, \ta_*^e),$$
    where $\ta_*^e$ is as in the theorem above.
    
\end{thm}
    In case the base field $k$ is of characteristic $0$ and $\a$ is a reduced principal ideal, 
    the first statement is a result by  Var\v{c}enco (for a proof, see \cite[Theorem 6.40]{ksc}).

    As the relation between MN Conjecture and ACC Conjecture is not yet proved in positive
    characteristic case, (i) in Theorem \ref{mainthm} alone does not imply ACC Conjecture
    for concerned pairs.
    By making use also of (ii) in Theorem \ref{mainthm}, we obtain the following ACC:  

\begin{cor}\label{accmld} Let $A$ be a smooth surface defined over an algebraically closed field of positive characteristic.
ACC Conjecture in the set of  pairs $(A, \a^e)$ holds for every DCC set $J$.
\end{cor}

\begin{cor}\label{acclct}  Let $A$,  $0\in A$ and  $J\subset \bR_{>0}$ be as above
          Then, the set 
     $$\left\{ \lct(0; A,\a^e) \mid  e\subset J\right\}$$ 
     satisfies ACC.
\end{cor}

   This statement is proved for $\bR$-Cartier divisors $\a_i$ in characteristic $p>5$ in 
   \cite[Theorem 1.10]{bir}.
   Note that in positive characteristic case,
   the ACC for ideals does not follow from ACC  for Cartier divisors.
   Actually, the reason why 
   for characteristic 0  we can reduce the problem into the problem for Cartier divisors is 
  because the equality $$\lct(0; A, \a^e)=\lct(0;A, (f)^e)$$
  holds
   for a general element $f_i\in \a_i $ and $(f)^e=(f_1)^{e_1}\cdots(f_s)^{e_s}$ by virtue of
    Bertini's theorem (generic smoothness)
   which does not hold in positive characteristic even for surfaces.

\begin{cor}\label{finite}
    Let $A$ be a smooth surface defined over an algebraically closed field of positive  
     characteristic and $0\in A$ a closed point.
     Then, for a fixed finite subset $e\subset \bR_{>0}$, the set
     $$\left\{ \mld(0; A, \a^e) \mid  \a^e \mbox{\ is\  a\  multiideal\  on\ } A\right\}$$
     is a finite set and coincides with the following set
     $$\left\{ \mld(0; \bA_\bC^2, \a^e) \mid  \a_i \mbox{\ is\  a\  monomial\ ideal\  on\ } \bA_\bC^2 \ 
     \mbox{for\ every}\ i\right\}.$$

\end{cor}

  The structure of this paper is as follows: 
  In the second section we give the definitions of two invariants and the proofs of the theorem and the corollaries.
   As we reduce the problem into that on the pairs $(\bA_k^2, \a^e)$,
   the calculation  $\ell_e$ for a given $e$ is a kind of combinatorics.
   We show some example of $\ell_e$ for some $e$.
   
{\bf Acknowledgement.} The author expresses her hearty thanks to Kohsuke Shibata
for his insightful  comments which improves the paper. 
She also would like to thank Masayuki Kawakita, Lawrence Ein and Mircea Musta\cedilla{t}\v{a} for the useful discussions. A big part of these discussions was  done during the author's stay in 
MSRI (Program: Birational Geometry and Moduli Theory) and she is grateful for the
support of MSRI.
The author would like to thank the referee for useful comments to improve the paper.

 \section{Preliminaries and the proofs}
\begin{defn} For a prime divisor $E$ over a non-singular variety $A$,
 let $\varphi: A'\to A$ be a proper birational morphism with normal $A'$
 such that  $E$ appears on $A'$.
  Let $k_E$ be the coefficient of the relative canonical divisor $K_{A'/A}$
  at $E$
  and $\val_E$ is the valuation defined by the divisor $E$.

A log discrepancy
  of the pair $(A,\a^e)$ is defined as 
  $$a(E; A, \a^e):=k_E-\sum_i e_i\val_E(\a_i)+1$$
  and the minimal log discrepancy of the pair at a closed point $0$ is defined as
  $$\mld(0; A, \a^e):=\inf \{a(E; A, \a^e)\mid E \ \mbox{prime\ divisor \ over\ }A\ \mbox{with\ 
  the \ center\ }0\}.$$
  The log canonical threshold of the pair at a closed point $0$ is defined as
  $$\lct(0, A, \a^e):=\inf\left. \left\{\frac{k_E+1}{\sum_i e_i\val_E(\a_i)}\right| 
   E \ \mbox{prime\ divisor \ over\ }A\ \mbox{with\ 
  the \ center\ containing}\ 0\right\}.$$ 
 \end{defn}
 \begin{lem}\label{weighted} 
   Let $A$ be a smooth surface defined over an algebraically closed field $k$ of 
   characteristic $p>0$ and $\a^e$ a multiideal with an exponent $e\subset \bR_{>0}$.
   Then, the following hold:
 \begin{enumerate}
  \item[$\rm(i)$] Every prime divisor computing $\mld(0; A, \a^e)\geq 0$ is obtained
  by a weighted blow up, and
   \item[$\rm(ii)$] There exists a prime divisor computing $\mld(0; A, \a^e)=-\infty$ 
   such that it is obtained by a weighted blow up.
 \end{enumerate}  
\end{lem}     

\begin{proof}
     For characteristic $0$,
    the statements are proved by Kawakita in \cite{kawk1}.
    Note that the only point in the proof he uses the characteristic $0$ is
    the ``Inversion of Adjunction'' of the form:\\
 \ \    Let $Q\in Y$ be a smooth point on a surface $Y$ and $F$ a smooth curve
    passing through $Q$, then \\
    (a) a triple $(Y, F, \a^e)$ is plt at $Q$, \ \ \ if and only if  \ \ \  
    (b) $(F, \a^e\o_F)$ is klt at $Q$.\\
As $Y$ is a surface, (a) is equivalent to:
 \begin{equation} \mld(Q; Y, I_F\a^e)>0,
 \end{equation} where $I_F$ is the defining ideal of $F$ on $Y$.
  On the other hand, as $F$ is a curve, (b) is equivalent to:
  \begin{equation} \mld(Q; F, \a^e\o_F)>0.
  \end{equation}
  In positive characteristic the Inversion of Adjunction of the following form is proved 
  in \cite{mj-p}:
  $$\mld(Q; Y, I_F\a^e)=\mld(Q; F, \a^e\o_F),$$
  which completes the equivalence of (a) and (b) for positive characteristic, and
  therefore completes the proof of the lemma for positive characteristic.
\end{proof}

\begin{lem}\label{toric} Let $k$ be an algebraically closed field of arbitrary characteristic.
    Then, we obtain 
\begin{enumerate}    
\item[$\rm(i)$] For every pair $(\bA_k^N, \a^e)$ with a monomial multiideal $\a^e$, 
  the multiideal on $\bA_\bC^N$ generated by the same monomial generators as of 
   $\a^e$ is denoted by  $ \widetilde{\a}^e$.
    Then we have
  $$\mld(0; \bA_k^N, \a^e)=\mld(0; \bA_\bC^N,\ta^e).$$
    
\item[$\rm(ii)$]   
    MN Conjecture holds in the set of the pairs $(\bA_k^N, \a^e)$ 
    for a monomial multiideal $\a^e$ 
   with an exponent $e\subset \bR_{>0}$.\end{enumerate}   
\end{lem}

\begin{proof} 
     Denote the maximal ideals of the origins in $\bA_k^N$ and
    $\bA_\bC^N$ by $\m_0$ and $\tilde\m_0$, respectively.
     The statement (i) follows from the fact that the both pairs $(\bA_k^N, \a^e\m_0)$
    and $(\bA_\bC^N,\ta^e\tilde\m_0)$ have toric log resolutions of the singularities, such that the  
    associated 
    fans are the same 
   and the valuation of the monomial ideals at the toric divisors corresponding to the same
    cone are the same.
    Therefore the minimal log discrepancies are the same and computed by  toric divisors 
    associated to the same cone. \\
 \ \    
     The statement (ii) is proved in \cite[Theorem 5.1]{mn} for characteristic $0$ and in \cite[Corollary 1.10]{ii}
     for positive characteristic.
     Here, we show a more direct proof than \cite[Corollary 1.10]{ii} by making use of
     the result of characteristic $0$.
     By Kawakita's result \cite{kawk}, the set of  mld's  with a fixed exponent
     $e$ is finite, if the base field is of characteristic $0$.
     By our statement (i), we also obtain that the set of 
      mld's  for our pairs $(\bA_k^N, \a^e)$ with a fixed exponent
     $e$ is finite for  the base field $k$ of positive characteristic. 
     Then, in the same way as in \cite[Theorem 5.1]{mn}, we can prove the statement for 
     positive characteristic.
 \end{proof}

{\it Proof of Theorem \ref{mainthm}}.
   Let $A$ be a smooth surface over the base field $k$ of arbitrary characteristic
   and fix a closed point $0\in A$.
   Let $E$ be a prime divisor computing $\mld(0; A,\a^e)$ as in Lemma \ref{weighted}.
   Then, 
   there are a regular system of parameters $x_1, x_2$ of $\o_{A,0}$
   and a pair of positive integers $w_1, w_2$
   such that the exceptional divisor $E$ obtained by the weighted blow up with
   respect to $  x_1, x_2$ with weight $w_1, w_2$ computes the minimal log 
   discrepancy $\mld(0; A,\a^e)$.
    Now we have a morphism 
    $$\rho:A\to \bA_k^2=\spec\ {k[x_1,x_2]},$$
    which is \'etale around the origin $0$ and have 
    the equalities $$\widehat{\o}_{A,0}=k[[x_1,x_2]]= \widehat{\o}_{\bA_k^2,0}.$$  
    Denote $\spec\ {k[[x_1,x_2]]}$ by $\wA$.
    As there are natural bijections:
    $$\left\{\begin{array}{c}
      \mbox{prime\ divisors\ over\ }A\\
      \mbox{with\ the\ center\ }$0$\\
      \end{array}\right\}\simeq
      \left\{\begin{array}{c}
      \mbox{prime\ divisors\ over\ }\wA\ \mbox{with}\\
      \mbox{the\ center\ at\ the\ closed\ point}\\
      \end{array}\right\}\simeq
      \left\{\begin{array}{c}
      \mbox{prime\ divisors\ over\ }\bA_k^2\\
      \mbox{with\ the\ center\ }$0$\\
      \end{array}\right\},
       $$
       we denote  prime divisors in these classes by the same symbol if those divisors correspond 
       to each other 
       under the above bijections.\\
       \ \ 
       We note that the value $k_E$ is preserved under these bijections.
       Let $\ha_i\subset k[[x_1, x_2]]$ be the extension of the ideals $\a_i\subset \o_A$ by $\rho*$
       and define $\ha^e:=\ha_1^{e_1}\cdots\ha_s^{e_s}$.
       Let $\ha_{*i}\subset k[[x_1, x_2]]$ be the ideal generated by all monomials appearing in 
       the elements of $\ha_i$.
       As the ring $ k[[x_1, x_2]]$ is Noetherian, $\ha_{*i}$ is generated by 
       finite number of monomials.
       Let those monomials generate an ideal $\a_{*i}$ in $k[x_1, x_2]$.
       Define $\ha_*^e$ and $\a_*^e$ as in the similar way as above.
   For a prime divisor $F$ over $A$ with the center $0$ (or over $\wA$  with the center at 
   the closed point, or over $\bA_k^2$ with the center at $0$),
   we obtain $\val_F  \a^e=\val_F\ha^e  \geq \val_F\ha_*^e=\val_F\a_*^e,$
   which implies 
\begin{equation}\label{F}  
  a(F; A, \a^e)=k_F+1-\val_F \a^e\leq k_F+1-\val_F \a_*^e=a(F; \bA_k^2,\a_*^e)
\end{equation}  
   and therefore
\begin{equation}\label{mono}   
   \mld(0; A, \a^e)\leq \mld(0; \bA_k^2,\a_*^e).
\end{equation}   
   
   Here, noting that the prime divisor $E$ is obtained by the weighted blow up,
   it follows that $\val_E$ is a monomial valuation over $\wA$, which yields
   equalities:
     $$\val_E  \a^e=\val_E\ha^e  = \val_E\ha_*^e=\val_E\a_*^e.$$
     Since $E$ computes $\mld(0; A, \a^e)$, we obtain either
     $$\mld(0; A, \a^e)=a(E; A, \a^e)=a(E; \bA_k^2,\a_*^e)\geq \mld(0; \bA_k^2,\a_*^e),
     $$
     $$ \mbox{or}\ \ \ \ \ \ \ \ \ 0>a(E; A, \a^e)=a(E; \bA_k^2,\a_*^e)\geq \mld(0; \bA_k^2,\a_*^e).$$
     By these inequalities and the opposite inequality (\ref{mono}), we have
\begin{equation}\label{=}  
  \mld(0; A, \a^e)=\mld(0; \bA_k^2,\a_*^e).
\end{equation}  
     By Lemma \ref{toric}, there is $\ell_{2,e}\in \bN$ depending only on $2$ and $e$ 
     such that there is a prime divisor $E'$ over $\bA_k^2$  with the center $0$
     computing $\mld(0; \bA_k^2,\a_*^e)$ and satisfying $k_{E'}\leq \ell_{2,e}$.
     Then, by (\ref{F}) and (\ref{=}) this divisor $E'$ also computes $\mld(0; A,\a^e)$,
     which completes the proof of (i).\\
     \ \  The proof of (ii) is clear from (i) in Lemma \ref{toric} and the above proof for (i) .
   $\Box$

  {\it Proof of Theorem \ref{weightedlct}.} 
    Let $A$ be a smooth surface over an  algebraically closed field $k$ of arbitrary characteristic,
    $\a^e$ a multiideal with a real exponent $e$ on $A$ and $0\in A$ a closed point.
    Let $t=\lct(0; A,\a^e)$, then the pair $(A, (\a^e)^t)$ is strictly log canonical.
    For the proof of Theorem, we may assume that there is an exceptional prime divisor $E$  computing
    the $\lct$.
    Then, $E$ has the center at $0$ and computes $\mld(0; A, (\a^e)^t)=0$.
    By Lemma \ref{weighted}, it follows that $E$ is obtained by a weighted blow up.
    For the second statement, let $t=\lct(0; A,\a^e)$ and it is computed by an exceptional divisor, 
    then as $A$ is a smooth surface, the center of the exceptional divisor is 0.
    Therefore  it follows $\mld(0; A, (\a^e)^t)=0$ and it is computed by the exceptional divisor.
    By the result (ii) in Theorem \ref{mainthm}, we have a monomial multiideal $\ta^e$
    with a real exponent $e$ on $\bA_\bC^2$ 
     such that
    $$0=\mld(0; A, (\a^e)^t)=\mld(0; \bA_\bC^2,(\ta^e)^t).$$
    And the proof of Theorem \ref{mainthm} also
    gives an exceptional divisor computing 
    $\mld(0; \bA_\bC^2,(\ta^e)^t)=0$.
    This shows that $t=\lct(0; \bA_\bC^2,\ta^e)$.
    $\Box$

    Note that  there is not necessarily an ``exceptional" prime divisor
     computing the log canonical threshold, although it is computed by some divisor 
    (may not be exceptional).
  Theorem \ref{weightedlct} states nothing for such a case.
  The following is an example:
\begin{exmp}
    Let $\a$ be an ideal defining a line $L$ in $\bA_k^2$. 
   Then, $\lct(0; \bA_k^2, \a)=1$ and the prime divisor computing $\lct(0; \bA_k^2, \a)$ is $L$ and   
   there is no exceptional 
   divisor over $\bA_k^2$ computing the $\lct$, although there is a divisor
   computing the $\lct$.
\end{exmp}

 {\it Proof of Corollary \ref{accmld}}.
       Let $A$ be a smooth surface over an algebraically closed field $k$ of positive characteristic
       and $0$   a closed point. 
       For every fixed DCC set $J \subset \bR_{>0}$, given a sequence:
       $$\mld(0; A, \a_{(1)}^{e_{(1)}}) < \mld(0; A, \a_{(2)}^{e_{(2)}})<\cdots
        $$
        such that $e_{(i)}\subset J$ for all $i=1,2,\ldots$.
      Then, by (ii) of Theorem \ref{mainthm}, this sequence coincides with 
      $$\mld(0; \bA_\bC^2, \ta_{(1)}^{e_{(1)}}) < \mld(0; \bA_\bC^2, \ta_{(2)}^{e_{(2)}})<\cdots,
        $$ 
   where 
    $\ta_{(i)}^{e_{(i)}}$'s are monomial multiideals on $\bA_\bC^2$.
     As ACC holds on the pairs over $\bC$, we obtain that the sequence stops at a finite stage.
$\Box$

 {\it Proof of Corollary \ref{acclct}}. 
   Let $A$ be a smooth surface over an algebraically closed field $k$ of positive characteristic
       and $0$   a closed point.  For every fixed DCC set $J \subset \bR_{>0}$, given a sequence:       
       $$\lct(0;A, {\a_{(1)}}^{e_{(1)}})<\lct(0;A, {\a_{(2)}}^{e_{(2)}})<\cdots$$
       such that $e_{(i)}\subset J$ for all $i=1,2,\ldots$.
       Here, we may assume that all lct are computed by exceptional divisors,
       because the sequence of lct's computed by non-exceptional divisors has ascending 
       chain condition.
       Now, by  the second statement of Corollary \ref{weightedlct},
       we obtain the ascending chain of lct's of pairs over $\bC$.
       Apply the result of ACC for characteristic $0$ (\cite{hmx}) to obtain the 
       sequence stops at a finite stage.
       $\Box$

 {\it Proof of Corollary \ref{finite}}.  
     The first statement follows in the same way as in the proof of Corollary \ref{accmld}
     by using (ii) of Theorm \ref{mainthm}.
     The second statement follows from the proof of the Theorem \ref{mainthm}.
     $\Box$

   As we reduce the problem into the one on the pairs of monomial ideals on $\bA_k^2$,
   we can calculate $\ell_e$ for a given $e$ by combinatorics.     
\begin{exmp} 
   First of all, note that $\mld(0; \bA_k^2, \a^e)$ and a toric divisor computing it are 
   determined by $e$ and the Newton 
   polygons $\Gamma(\a_i)$ $(i=1,\ldots, s)$
   (for the definition of the Newton polygon, the reader can refer to \cite[Definition 5.2]{findet}).
   Actually, we have
   $$\mld(0; \bA_k^2, \a^e)=\inf \left\{\langle \bp, {\bf 1}\rangle-\sum_i {e_i}\langle \bp,\Gamma(\a_i)\rangle \right\},$$
   where $\bp=(p_1,p_2)$ runs whole in $\bN^2$ in the right hand side and
   $\langle \bp,\Gamma(\a_i)\rangle:=\min \{\langle \bp,\bq\rangle \mid \bq\in \Gamma(\a_i)\}$.
   A toric divisor computing the $\mld\geq 0$ is a divisor $E_{\bp}$, where $\bp$ attains the
   infimum in the right hand side.
   In the following, we will estimate $\ell_e$ for a given $e$.
   Note that we check only toric divisors and obtain $\ell_e$,
   therefore it may not be optimal,
   {\it i.e.,} there may be non-toric divisor $E$ computing the mld with smaller $k_E$.\\
     \ \ For every $e$, if $\a_i={\mathcal O}_{\bA_k^2}$ ($i=1,\ldots, s$), then 
    $\mld(0; \bA_k^2, \a^e)=2$ and computed by $E_{\bf 1}$,
    where  ${\bf 1}=(1,1)$. 
    So, in the following, we exclude this trivial case.
   \begin{enumerate}
     \item  Case $\# e=1$. Let $e=\{e_1\}$ $(e_1\in \bR_{>0})$.
     \begin{enumerate}
        \item When $e_1> 2$, it follows that $a(E_{\bf1};\bA^2, \a^e)=2-e_1<0$,
        therefore it is obvious that $\mld(0;\bA^2, \a^e)=-\infty$ and it is computed by $E_{\bf 1}$.
        Therefore, $\ell_e=1$.
        \item  When $1<e_1\leq 2$, either $\mld(0;\bA^2, \a^e)=2-e_1$ or $\mld(0;\bA^2, \a^e)=-\infty$, the first case is computed by $E_{\bf 1}$ and in the second case an upper bound of minimal $k_E$ such that
        $E$ computes the mld is $\left[ \frac{1}{e_1-1}\right]+1$. 
        In this case $E=E_\bp$ $(\bp=\left(\left[ \frac{1}{e_1-1}\right]+1, 1\right))$  and the ideal $\a$ is 
        generated by $x$.
        Therefore, we obtain $\ell_e=\left[ \frac{1}{e_1-1}\right]+1$.
        
        This is proved as follows:
        
        {\bf Case 1} $\mult_0\a\geq 2$.
        
        In this case the Newton polygon $\Gamma(\a)\subset \widetilde\Gamma$,
        where $\widetilde\Gamma$ is the convex hull of 
        $$\left((2, 0)+\bR_{\geq0}^2 \right)\cup \left((0, 2)+\bR_{\geq0}^2 \right).$$
        As $e_1>1$, we have $e_1\Gamma(\a)\subset e_1\widetilde\Gamma$ and
        ${\bf1}\not\in e_1\widetilde\Gamma$ which implies 
        $$a(E_{\bf1}; \bA^2, \a^e)=2-e_1\langle {\bf1}, \Gamma(\a)\rangle
        <2-e_1\langle {\bf1}, \widetilde\Gamma\rangle< 0.$$
        Therefore, in this case $\mld=-\infty$ and it is computed by $E_{\bf1}$.
        
        {\bf Case 2} $\mult_0\a=1$.
        
        In this case, the Newton polygon is either
        the convex hull of 
        \begin{equation}\label{diagonal}
        \left((1, 0)+\bR_{\geq0}^2 \right)\cup \left((0, 1)+\bR_{\geq0}^2 \right)
        \end{equation}
        or the convex hull of 
        \begin{equation}\label{vertical}
        \left((1, 0)+\bR_{\geq0}^2 \right)\ \ \ \mbox{or}\ \ \ \ 
           \left((0, 1)+\bR_{\geq0}^2 \right).
        \end{equation}   
        
        In the case (\ref{diagonal}), $a(E_{\bp};\bA^2, \a^e)=\langle {\bp},{\bf 1}\rangle
        -e_1\langle {\bp},\Gamma(\a)\rangle$ and it is minimized by $\bp=\bf 1$.
        It says that $E_{\bf 1}$ computes $\mld=2-e_1$.
        
        In the case (\ref{vertical}) we can only prove the first case,
        as these two are symmetric.
        As $e_1>1$, we have ${\bf 1}\not\in e_1\Gamma(\a)$  which means 
        $\mld=-\infty$.
        An integer vector $\bp$ such that $a(E_{\bp};\bA^2, \a^e)=\langle{\bp}, {\bf 1}\rangle
        -e_1\langle{\bp}, \Gamma(\a)\rangle<0$ minimizes 
        $ \langle{\bp}, {\bf 1}\rangle=p_1+p_2$
        is 
        $$p_1=\left[\frac{1}{e_1-1}\right]\ \ \  \mbox{and}\ \ \ p_2=1,$$
         where $p_1$ and $p_2$ are the coordinates
        of $\bp$.

        \item
        When $e_1=1$, $\mld(0;\bA^2, \a^e)=1$ or $0$ or $-\infty$ and
        is computed by $E_{\bf 1}$ for the former two cases.
        For the case $\mld(0;\bA^2, \a^e)=-\infty$ an upper bound of minimal $k_E$
        such that $E$ computes the mld 
        is $4$.
                In this case $E=E_{(3,2)}$ and the ideal $\a$ is generated by $x^2$ and $y^3$.
        Therefore, we obtain $\ell_{\{1\}}=4$.
        The proof is similar to the previous case
         by  divided into the cases according to the multiplicity of $\a$,
         so we omit the proof.
        \item 
        When $e_1<1$, for a smaller $e_1$ we have more cases to be checked to get $\ell_e$.
        As an example, we consider $e=\{1/2\}$, then we have $\mld(0;\bA^2, \a^e)=3/2$ or $1$ or $1/2$ or $0$ or $-\infty$
       and it is computed by $E_{\bf 1}$ for the former four cases. 
         For the case $\mld(0;\bA^2, \a^e)=-\infty$ an upper bound of the minimal $k_E$
         such that $E$ computes the mld
        is $9$.
        In this case $E=E_{(7,3)}$ and the ideal $\a$ is generated by $x^3$ and $y^7$.
        Therefore, we obtain $\ell_{\{1/2\}}=9$.
      
       An example for non rational $e_1$ is as follows:
       
        Let $e_1=2/\pi$, where $\pi$ is the circular constant.
        Then, the value of $\mld(0; \bA_k^2, \a^e)=2-2/\pi$ or $2-4/\pi$ or $2-6/\pi$ or $-\infty$.
        And an  upper bound of the minimal $k_E$ such that $E$ computes the mld is
        $6$. In this case $E=E_{(4,3)}$ and the ideal  $\a$ is generated by $x^3$ and $y^4$.
        Therefore, we obtain $\ell_{\{2/\pi\}}=6$.
      \end{enumerate}     
      \item Case $\#e=2$. 
      We consider $e=(1, 1/2)$. In this case, the possible values of  $\mld(0;\bA^2, \a^e)$
      are $3/2$ or $1$ or $1/2$ or $0$ or $-\infty$. 
      An upper bound of minimal $k_{E_\bp}$ such that $E_\bp$ computes $\mld(0;\bA^2, \a^e)$
      is $9$ and in this case $E=E_{(7,3)}$, $\a_1={\mathcal{O}}_{\bA^2}$ and 
      $\a_2$ is generated by $x^3$ and $y^7$.
      \end{enumerate}
\end{exmp}    
  \begin{rem}   If $e' \subset e \subset \bR_{>0}$, then we have $\ell_{e'}\leq \ell_e$,
  where we assume that $\ell_{e'}$ and $ \ell_e$ are optimal.
  This is because every $\a^{e'}$ can be written as $\b^e$ by
  $\b_i=\a_i$ for $e_i\in e'$ and $\b_i={\mathcal{O}}_{\bA^2}$ otherwise.
\end{rem}
\makeatletter \renewcommand{\@biblabel}[1]{\hfill#1.}\makeatother

\noindent  YMSC, Tsinghua University, Beijing/ 
Graduate School of Math. Sci., the University of Tokyo, \\
shihokoishii@mac.com/
shihoko@ms.u-tokyo.ac.jp
\end{document}